\documentclass[11pt]{amsart}
\usepackage{amsmath, enumerate}
\usepackage{amsfonts, amssymb, amsthm}
\usepackage{mathtools}
\usepackage{geometry}
\usepackage{a4wide}

\numberwithin{equation}{section}
\theoremstyle{plain}

\newtheorem{theorem}{Theorem}[section]
\newtheorem{lemma}[theorem]{Lemma}
\newtheorem{proposition}[theorem]{Proposition}
\newtheorem{corollary}[theorem]{Corollary}

\theoremstyle{definition}
\newtheorem{definition}[theorem]{Definition}

\newtheorem{example}[theorem]{Example}

\theoremstyle{remark}
\newtheorem{remark}[theorem]{Remark}

\usepackage{pst-node}
\usepackage{auto-pst-pdf}
\usepackage{tikz-cd}

\newcommand{\R}{\mathbb{R}}
\newcommand{\N}{\mathbb{N}}
\newcommand{\Z}{\mathbb{Z}}
\newcommand{\Q}{\mathbb{Q}}
\newcommand{\C}{\mathbb{C}}

\renewcommand{\H}{\mathbb{H}}

\newcommand{\PSL}{{\mathrm{PSL}}}

\newcommand{\tr}{\mathrm{tr}}

\renewcommand{\Re}{\mathrm{Re}}

\begin{document}
\title{Twisted traces of modular functions on hyperbolic $3$-space}

\author{S.~Herrero}
\address{Universidad de Santiago de Chile, Dept.~de Matem\'atica y Ciencia de la Computaci\'on, Av.~Libertador Bernardo O'Higgins 3363, Santiago, Chile, and ETH, Mathematics Dept., CH-8092, Z\"urich, Switzerland}
\email{sebastian.herrero.m@gmail.com}
\thanks{S.~Herrero's research is supported by ANID/CONICYT FONDECYT Iniciaci\'on grant 11220567 and by SNF grant CRSK-2{\_}220746}

\author{\"O.~Imamo\=glu}
\address{ETH, Mathematics Dept., CH-8092, Z\"urich, Switzerland}
\email{ozlem@math.ethz.ch}
\thanks{\"O.~Imamo\=glu's research  is supported by SNF grant 200021-185014}

\author{A.-M.~von Pippich}
\address{
University of Konstanz, Dept. of Mathematics and Statistics, Universit\"atsstra{\ss}e 10, 78464 Konstanz, Germany}
\email{anna.pippich@uni-konstanz.de}
\thanks{A.-M.~von Pippich's research is supported by the LOEWE research unit \emph{Uniformized Structures in Arithmetic and Geometry}}

\author{M.~Schwagenscheidt}
\address{ETH, Mathematics Dept., CH-8092, Z\"urich, Switzerland}
\email{mschwagen@ethz.ch}
\thanks{M.~Schwagenscheidt's research is supported by SNF grant PZ00P2{\_}202210}

\begin{abstract}
  	We compute analogues of twisted traces of CM values of harmonic modular functions on hyperbolic $3$-space and show that they are essentially given by Fourier coefficients of the $j$-invariant. From this we deduce that the twisted traces of these harmonic modular functions are integers. Additionally, we compute the twisted traces of Eisenstein series on hyperbolic $3$-space in terms of Dirichlet $L$-functions and divisor sums.
\end{abstract}


 
\maketitle

\section{Introduction and statement of results}

The values of the modular $j$-invariant at imaginary quadratic irrationalities are called \emph{singular moduli}. By the theory of complex multiplication, the singular moduli are algebraic integers, and their traces are rational integers. A famous result of Zagier \cite{zagiertraces} states that the generating function of traces of singular moduli is a weakly holomorphic modular form of weight $3/2$. This result has been extended in various directions, for example to traces of (weakly holomorphic and non-holomorphic) modular forms on congruence subgroups by Bruinier and Funke \cite{bruinierfunke}. 
Mizuno \cite{mizuno} proved an analogue of Zagier's result for automorphic forms on hyperbolic $3$-space $\H^3$. He showed that the ``traces'' of Niebur Poincar\'e series and Eisenstein series on $\H^3$ appear in the Fourier coefficients of certain linear combinations of non-holomorphic elliptic Poincar\'e series and Eisenstein series, respectively, of odd weight. Kumar \cite{kumar} gave ``twisted'' versions of Mizuno's results. Similar results in the untwisted case were obtained by Matthes \cite{matthes} on higher dimensional hyperbolic spaces. 


In the present paper we restrict our attention to the special case of harmonic automorphic forms on $\H^3$ and their ``twisted traces''. The automorphic forms we consider here are constructed as special values of Niebur Poincar\'e series, 
which can be viewed as analogues on $\H^3$ of the modular $j$-invariant. 
We show that the twisted traces of these harmonic automorphic forms on $\H^3$ are Fourier coefficients of weakly holomorphic modular functions for $\PSL_2(\Z)$ (by ``Zagier duality'', they are also Fourier coefficients of weight 2 weakly holomorphic modular forms).
This allows us to deduce that the twisted traces are integers, which is not a priori clear since the theory of complex multiplication is not available in our setting. Using the same method we can also compute the twisted traces of weight 0 real-analytic Eisenstein series on $\H^3$.

In the following we describe our results in more detail. Let $\Q(\sqrt{D})$ be an imaginary quadratic field of discriminant $D < 0$, and let $\mathcal{O}_D$ be its ring of integers. The group 
\[
\Gamma = \PSL_2(\mathcal{O}_D)
\]
acts on the hyperbolic $3$-space
\[
\H^3 = \{P = z + r j \, : \, z \in \C, r \in \R_{> 0}\}
\] 
(viewed as a subset of $\R[i,j,k]$ of Hamilton's quaternions) by fractional linear transformations. Let $\mathfrak{d}_D^{-1}$ denote the inverse different of $\Q(\sqrt{D})$. For $\nu \in \mathfrak{d}_D^{-1}$ with $\nu \neq 0$  and $s\in \C$ with $\Re(s) > 1$ we consider the Niebur Poincar\'e series
\[
F_{\nu}(P,s) = 2\pi |\nu|\sum_{\gamma \in \Gamma_\infty \backslash \Gamma}r(\gamma P)I_s(4\pi |\nu|r(\gamma P))e(\tr(\nu z(\gamma P))),
\]
with the usual $I$-Bessel function. Here $\Gamma_\infty$ denotes the subgroup of $\Gamma$ consisting of the matrices $\pm\left(\begin{smallmatrix}1 & \beta \\ 0 & 1 \end{smallmatrix}\right)$ with $\beta \in \mathcal{O}_D$, and we write $\gamma P = z(\gamma P) + r(\gamma P)j$. The Niebur Poincar\'e series converges for $\Re(s) > 1$ and satisfies the Laplace equation
\begin{align}\label{eq laplace}
(\Delta - (1-s^2))F_{\nu}(P,s) = 0,
\end{align}
where $\Delta$ denotes the usual hyperbolic Laplace operator on $\H^3$. By analyzing the Fourier expansion, it was shown in \cite[Proposition~4.4]{herreroimamogluvonpippichtoth} that $F_{\nu}(P,s)$ has an analytic continuation to $\Re(s) > 1/2$ which is holomorphic at $s = 1$. 

\begin{definition}
For $\nu \in \mathfrak{d}_D^{-1}$ with $\nu \neq 0$ we define the function
\[
J_{\nu}(P) = F_{\nu}(P,1),
\]
which is a harmonic $\Gamma$-invariant function on $\H^3$.
\end{definition}

We may view the functions $J_{\nu}(P)$ as hyperbolic $3$-space analogues of the elliptic modular functions $j_n(\tau)$ whose Fourier expansions are of the form $j_n(\tau)= q^{-n} + O(q)$. Indeed, $j_n(\tau)$ can be constructed as a Niebur Poincar\'e series in an analogous way, see Section~\ref{section niebur}. This inspired the notation $J_{\nu}(P)$.

We are interested in the values of $J_{\nu}(P)$ at special points in $\H^3$. For a positive integer $m > 0$ we let
\[
L_{|D|m}^+ = \left\{X = \begin{pmatrix}a & b \\ \overline{b} & c \end{pmatrix} \, :\, a,c \in \N,\, b \in \mathcal{O}_D, \, \det(X)=|D|m \right\}
\]
be the set of positive definite integral binary hermitian forms of determinant $|D|m$ over $\Q(\sqrt{D})$. The group $\Gamma$ acts on $L_{|D|m}^+$, with finitely many orbits. To $X = \left(\begin{smallmatrix}a & b \\ \overline{b} & c \end{smallmatrix} \right) \in L_{|D|m}^+$ we associate the \emph{special point}
\[
P_X = \frac{b}{c} + \frac{\sqrt{|D|m}}{c}j \in \H^3.
\]
It can be viewed as an analogue of a CM point on $\H^3$. We refer to the book \cite{elstrodt} for more on binary hermitian forms and the connection to the hyperbolic $3$-space.

We now define twisted traces of special values of functions on $\H^3$. We can write $D = \prod_{p \mid D} p^*$ as a product of \emph{prime discriminants} $p^*$, that is, for an odd prime $p$ we have $p^* = (\frac{-1}{p})p$, and we have $2^* \in \{-4,\pm 8\}$. For $X \in L_{|D|m}^+$ we define the \emph{twisting function}
\[
\chi_D(X) = \prod_{p \mid D}\chi_{p^*}(X), \qquad \chi_{p^*}(X) = \begin{cases}
\left( \frac{p^*}{a}\right), & \text{if } p \nmid a, \\
\left( \frac{p^*}{c}\right), & \text{if } p \nmid c, \\
0, & \text{if } p \mid \text{g.c.d.}(a,c).
\end{cases}
\]
This function was considered by Bruinier and Yang \cite{bruinieryang} and Ehlen \cite{ehlen} for positive fundamental discriminants $D$, and it was used to study twisted Borcherds products on Hilbert modular surfaces. 

As in \cite{bruinieryang}, one can check that $\chi_D(X)$ is well-defined and $\Gamma$-invariant. For a $\Gamma$-invariant function $f$ on $\H^3$ we define its \emph{$m$-th twisted trace} by
\[
\tr_{m,D}(f) = \sum_{X \in \Gamma \backslash L_{|D|m}^+}\frac{\chi_D(X)}{|\Gamma_X|}f(P_X),
\]
where $\Gamma_X$ denotes the stabilizer of $X$ in $\Gamma$. 

\begin{remark}\label{rmk:genus_characters}
Note that we twist with the quadratic character associated to the discriminant $D$ of the underlying imaginary quadratic field $\Q(\sqrt{D})$. In contrast, the classical traces of CM values of modular functions for $\PSL_2(\Z)$ can be twisted by genus characters associated to arbitrary fundamental discriminants dividing the discriminant of the CM points (see \cite{zagiertraces}, for example). It would be interesting to find an analogue of ``twists by genus characters'' of traces on $\H^3$.
\end{remark}

Our main result is the following explicit evaluation of the twisted traces of the harmonic modular functions $J_{\nu}$ in terms of Fourier coefficients of the elliptic modular functions $j_n$.

\begin{theorem}\label{theorem trace niebur}
For $\nu \in \mathfrak{d}_D^{-1}$ with $\nu \neq 0$ we have
\[
\tr_{m,D}(J_{\nu}) = m\sum_{d \mid \nu}\left( \frac{D}{d}\right)d \, c_{|D||\nu|^2/d^2}(m),
\]
where $c_{n}(m)$ are the Fourier coefficients of $j_n(\tau)$, and $d \mid \nu$ means that $\nu/d \in \mathfrak{d}_D^{-1}$.
\end{theorem}

\begin{remark}\label{rmk: dual basis int weight}
\begin{enumerate}
    \item The above theorem implies that for every unit~$u\in \mathcal{O}_D^{\times}$ and every~$\nu \in \mathfrak{d}_D^{-1}$ with~$\nu\neq 0$ we have~$\tr_{m,D}(J_{u\nu})=\tr_{m,D}(J_{\nu})$. This can also be proved directly as follows. For every~$P=z+rj\in \H^3$, we have the relation
$$J_{u\nu}(z+rj)=J_{\nu}(uz+rj).$$
Moreover, if~$P=P_X$ is the special point associated to the positive definite hermitian form~$X=\left(\begin{smallmatrix}a & b \\ \overline{b} & c \end{smallmatrix} \right)$, then~$uz+j$ is the special point associated to~$X_u=\left(\begin{smallmatrix}a & ub \\ \overline{ub} & c \end{smallmatrix} \right)$ and~$\chi_D(X)=\chi_D(X_u)$. The equality of the traces is then a consequence of the fact that~$X\mapsto X_u$ defines a bijection~$L^+_{|D|m}\to L^+_{|D|m}$.
\item    The family~$\{1\}\cup \{j_n\}_{n\geq 1}$ is a basis of the~$\C$-vector  space~$M_0^{!}$ of weakly holomorphic modular functions for~$\PSL_2(\Z)$. Defining
    $$S_n:=-\frac{j_n'}{n}=\frac{1}{q^n}+\sum_{m=1}^{\infty}b_n(m)q^m,$$
    where~$j_n'(\tau) = \frac{1}{2 \pi i}\frac{\partial}{\partial \tau}j_n(\tau)$, we get a family a functions~$\{S_n\}_{n\geq 1}$ that is a basis of the space~$M_2^{!}$ of weight~$2$ weakly holomorphic modular forms for the same group. The families~$\{j_n\}_{n\geq 1}$ and~$\{S_n\}_{n\geq 1}$ are \emph{Zagier dual} in the sense that
    $$c_m(n)=-b_n(m) \text{ for all integers }n,m\geq 1,$$
    and the formula in Theorem~\ref{theorem trace niebur} can be rewritten as
    \begin{equation*}\label{eq:traces_as_sum_of_coeff_of_k_m}
    \tr_{m,D}(J_{\nu}) = -m \sum_{d \mid \nu}\left( \frac{D}{d}\right)d \, b_{m}\left(\frac{|D||\nu|^2}{d^2}\right).    
    \end{equation*}
\end{enumerate}
    \end{remark}

The method we use for  the  proof of Theorem~\ref{theorem trace niebur} is by now classical and has been employed in various previous works such as \cite{zagierdoinaganuma}, \cite{kohnen}, \cite{iwaniec}, \cite{dukeimamoglutoth}, \cite{mizuno} and~\cite{kumar}, among others. We write out the left-hand side as an infinite series involving an $I$-Bessel function and certain finite exponential sums. By writing $j_n$ as a special value of a Niebur Poincar\'e series for $\PSL_2(\Z)$, one obtains an expression for its coefficients $c_{n}(m)$ as an infinite series involving the same $I$-Bessel function and certain Kloosterman sums, see Section~\ref{section niebur}. Hence, the proof of the theorem boils down to an identity of finite exponential sums, see Section~\ref{section exponential sums}. The details of the proof will be given in Section~\ref{section proofs}.

Since the coefficients of $j_n$ are integers, we obtain the following rationality result.

\begin{corollary}
	The twisted traces $\tr_{m,D}(J_{\nu})$ are integers which are divisible by $m$.
\end{corollary}

This is remarkable since the theory of complex multiplication is not available in the hyperbolic $3$-space setting. We can rephrase the evaluation of $\tr_{m,D}(J_{\nu})$ as a modularity result. Recall that $j_n'(\tau) = \frac{1}{2 \pi i}\frac{\partial}{\partial \tau}j_n(\tau)$.


\begin{corollary}\label{cor: generating series over m}
	 For~$\nu \in \mathfrak{d}_D^{-1}$ with~$\nu\neq 0$, let  
 \begin{equation*}
 \mathcal{Z}_{\nu,D}(\tau)= -\sum_{d \mid \nu}\left( \frac{D}{d}\right)\frac{|D||\nu|^2}{d}q^{-|D||\nu|^2/d^2} + \sum_{m =1}^\infty \tr_{m,D}(J_{\nu})q^m    
 \end{equation*}
  be  the generating function over~$m$ for the twisted traces~$\tr_{m,D}(J_{\nu})$.  Then 
 \begin{equation}\label{eq: generating function over m traces}
	  \mathcal{Z}_{\nu,D}(\tau)= \sum_{d \mid \nu}\left( \frac{D}{d}\right)d \, j_{|D||\nu|^2/d^2}'(\tau).  
	\end{equation}
	In particular, it is a weakly holomorphic modular form of weight $2$ for $\PSL_2(\Z)$. 
\end{corollary}

\begin{example}
	Let $D  = -4$, so $\Q(\sqrt{D}) = \Q(i)$. Then $\mathcal{O}_D = \Z[i]$ and $\mathfrak{d}_D^{-1} = \frac{1}{2}\Z[i]$. Take $\nu = \frac{1}{2}$, such that $|\nu|^2 = \frac{1}{4}$. Then for every $m\geq 1$ we have
	\[
	\tr_{m,-4}(J_{1/2}) = m c_1(m),
	\]
	where $c_1(m)$ is the $m$-th coefficient of the $j$-invariant.	In particular, we have
	\[
	-q^{-1} + \sum_{m=1}^{\infty}\tr_{m,-4}(J_{1/2})q^m = j'(\tau).
	\]
    For example, for $m = 1$ the set $\Gamma \backslash L_{4}^+$ consists of the three binary hermitian forms $X_1 = \left(\begin{smallmatrix}4 & 0 \\ 0 & 1 \end{smallmatrix} \right), X_2 = \left(\begin{smallmatrix}3 & 1+i \\ 1-i & 2 \end{smallmatrix} \right), X_3 = \left(\begin{smallmatrix}2 & 0 \\ 0 & 2 \end{smallmatrix} \right)$, with corresponding special points $P_1 = 2j, P_2 = \frac{1+i}{2} + j, P_3 = j$, and stabilizers of orders $4,4,8$, respectively (see~\cite[p.~413]{elstrodt}). Moreover, we have $\chi_D(X_1) = 1, \chi_D(X_2) = -1$, and $\chi_D(X_3) = 0$. By evaluating the defining series of $J_{1/2}$ numerically, we can now compute
    \[
    \tr_{1,-4}(J_{1/2}) = \frac{1}{4}J_{1/2}(P_1) - \frac{1}{4}J_{1/2}(P_2) = \frac{1}{4}\cdot (786286.36 \dots) - \frac{1}{4}\cdot(-1249.60\dots ) = 196883.99\dots
    \]
    Up to rounding errors, this agrees with the coefficient $c_1(1) = 196884$. 
\end{example}

\begin{remark}
Using the basis~$\{S_n\}_{n\geq 1}$ from Remark~\ref{rmk: dual basis int weight}(2) we can rewrite~\eqref{eq: generating function over m traces} 
$$ -\sum_{d \mid \nu}\left( \frac{D}{d}\right)\frac{|D||\nu|^2}{d}q^{-|D||\nu|^2/d^2} + \sum_{m =1}^\infty \tr_{m,D}(J_{\nu})q^m    =-\sum_{d \mid \nu}\left( \frac{D}{d}\right)\frac{|D||\nu|^2}{d}\, S_{|D||\nu|^2/d^2}(\tau).$$
This can be compared with Zagier's result~\cite{zagiertraces} in the elliptic case. 
Indeed, in \emph{loc.~cit.} Zagier constructs a basis~$\{g_D\}_{D>0,D\equiv 0,1(4)}$ for the space~$M^{!,+}_{3/2}$ of weight~$3/2$ weakly holomorphic modular forms in the plus space, and for every integer~$n\geq 1$ and every fundamental discriminant~$D>0$ he shows the equality 
$$-\sum_{d\mid n}\left(\frac{D}{d}\right)\frac{\sqrt{D}n}{d}q^{-Dn^2/d^2} +\sum_{d=1}^\infty \tr_{-d,D}(j_n)q^d=-\sum_{d\mid n}\left(\frac{D}{d}\right)\frac{\sqrt{D}n}{d}\, g_{Dn^2/d^2}(\tau)$$
for the generating function of the twisted traces of the function~$j_n$ over CM points of discriminant~$-dD$.
\end{remark}

\begin{remark}
\begin{enumerate}
    \item     On the one hand, the results of Mizuno \cite{mizuno} imply that the non-twisted traces of $J_\nu$ are Fourier coefficients of a \emph{non-holomorphic} modular form of odd weight for $\Gamma_0(|D|)$~with character~$\left(\frac{D}{\cdot}\right)$. Indeed, the form~$\mathcal{G}(z,1)$  obtained from~\cite[Theorem~6]{mizuno} has weight~$k$ an odd integer, and it is an eigenfunction of the  Laplace operator~$\Delta_k=-v^2\left(\frac{\partial^2}{\partial_u^2}+\frac{\partial^2}{\partial_v^2}\right)+ikv\left(\frac{\partial}{\partial_u}+i\frac{\partial}{\partial_v}\right)$ with eigenvalue~$(k^2-2k)/4\neq 0$. 
    Hence, it seems that the non-twisted traces of $J_\nu$ do not have good algebraic properties. On the other hand, using vector-valued modular forms and similar methods as in \cite{mizuno},  one can prove some rationality results for the non-twisted traces of the \emph{non-harmonic} function~$F_{\nu}(P,k/2)$, by expressing them in terms of the Fourier coefficients of weakly holomorphic modular forms. 
    For example, in the case~$D=-4$ and~$\nu=\frac{1}{2}$, noting that there is only one form in $ \Gamma \backslash L_{1}^+$, corresponding to the special point $P=j$,  with stabilizer of order~$4$,  we numerically compute to get
    $$\tr_{1}\left(F_{1/2}\left(\cdot,\tfrac{3}{2}\right)\right)=\frac{1}{4}F_{1/2}\left(j,\tfrac{3}{2}\right)=384.$$
    \item Kumar \cite{kumar} gave twisted versions of Mizuno's results, but did not consider the relation to the $j$-invariant or the algebraic nature of the twisted traces of $J_\nu$.
    \end{enumerate}
\end{remark}

It is also possible to construct a  generating function summing over~$\nu$ for the twisted traces~$\tr_{m,D}(J_\nu)$. This leads to a weight 2 automorphic form on~$\H^3$ with singularities at the special points in~$T_{m,D}=\{P_X: X\in L^+_{m|D|}, \chi_D(X)\neq 0\}$. 
In order to be more precise, let~$m>0$ be a fixed integer, and for $P=z+rj$ and~$\ell \in \{0,1,2\}$ let us define
\begin{equation}\label{eq F_i}
\mathcal{F}_{m,D}^{(\ell)}(z+rj)=
\sum_{\substack{\nu \in \mathfrak{d}^{-1}_D \\ \nu\neq 0}}\tr_{m,D}(J_{\nu}) \xi(\nu)^{\ell-1}r\widetilde{K}_{\ell}(4\pi |\nu|r)e(\tr(\nu z)),
\end{equation}
with~$\xi(\nu)=\frac{\nu}{|\nu|}$ 
and special functions
\[
\widetilde{K}_\ell(y) = 
 \begin{dcases}
-K_1(y) , & \ell=2,  \\
 2iK_0(y) ,  &  \ell=1,\\
 K_1(y) , & \ell=0.
\end{dcases}
\]
The series~\eqref{eq F_i} converge for~$r>\sqrt{m|D|}$, and in Section~\ref{sec: generating functions over nu} we prove the following.

\begin{theorem}\label{thm generating function over nu}
Let~$m>0$ be an integer. Then, for each~$\ell\in \{0,1,2\}$ the function~$\mathcal{F}_{m,D}^{(\ell)}$ given in~\eqref{eq F_i} extends to a smooth function on~$\H^3\setminus T_{m,D}$, and $\mathcal{F}_{m,D}=(\mathcal{F}_{m,D}^{(2)},\mathcal{F}_{m,D}^{(1)},\mathcal{F}_{m,D}^{(0)})^t$ defines an automorphic form of weight~$2$ for~$\Gamma$ on~$\H^3\setminus T_{m,D}$.
\end{theorem}

The proof of Theorem~\ref{thm generating function over nu} is based on properties of the automorphic Green's function associated to~$\Gamma=\PSL(\mathcal{O}_D)$ 
and properties of certain differential operators acting on automorphic forms on~$\H^3$ studied by Friedberg in~\cite{friedberg}, see Section~\ref{sec: generating functions over nu}.

\begin{remark}
    The form~$\mathcal{F}_{m,D}$ in Theorem~\ref{thm generating function over nu} is the analogue of the generating function
\[
\tr_{-d,D}\left(\frac{j'(\tau)}{j(\cdot)-j(\tau)}\right)=\sum_{n=1}^{\infty}\tr_{-d,D}(j_n)q^n,
\]
    which is an elliptic modular form of weight~2 for~$\PSL_2(\Z)$ with a singularity at each CM point of discriminant~$-dD$ whenever the twisting function does not vanish. 
\end{remark}

Finally, we compute the twisted trace of the weight $0$ real-analytic Eisenstein series
\[
E_0(P,s) = \sum_{\gamma \in\Gamma_\infty \backslash \Gamma}r(\gamma P)^{s+1}.
\]
The Eisenstein series converges absolutely for $s\in \C$ with $\Re(s) > 1$, and has meromorphic continuation to $\C$ with a simple pole at $s = 1$, whose residue is independent of $P$. It also satisfies the Laplace equation \eqref{eq laplace}. Moreover, it defines a $\Gamma$-invariant function on $\H^3$.

\begin{theorem}\label{theorem trace eisenstein}
	The $m$-th twisted trace of the Eisenstein series is given by
	\[
	\tr_{m,D}\left(E_0(\,\cdot\,,s)\right) = \frac{|D|^{\frac{s+1}{2}}L_D(s)}{\zeta(s+1)}m^{\frac{1-s}{2}}\sigma_s(m) ,
	\]
	where $\sigma_s(m) = \sum_{d \mid m}d^s$ is a divisor sum and $L_D(s) = \sum_{n =1}^\infty \left( \frac{D}{n}\right)n^{-s}$ is a Dirichlet $L$-series. In particular, the twisted trace is holomorphic at $s = 1$, with
	\[
	\tr_{m,D}\left(E_0(\,\cdot\,,s)\right)\big|_{s = 1} = \frac{12}{\pi}\frac{\sqrt{|D|}h(D)}{w(D)}\sigma_1(m),
	\]
	where $h(D)$ is the class number of $\Q(\sqrt{D})$ and $\omega(D) \in \{1,2,3\}$ is half the number of units of~$\mathcal{O}_D$.
\end{theorem}

The proof of Theorem~\ref{theorem trace eisenstein} is analogous to the proof of Theorem~\ref{theorem trace niebur}, compare Sections~\ref{section proofs} and~\ref{section proofs 2}. The holomorphicity of the twisted traces of the Eisenstein series at $s = 1$ and the fact that the residue of the Eisenstein series does not depend on $P$ together imply that the twisted traces of the constant $1$ function vanish.

\begin{corollary} 
	We have
	\[
	\tr_{m,D}(1) = 0.
	\]
\end{corollary}

\begin{remark}
    An analogous vanishing result for twisted Hurwitz class numbers (that is, twisted traces of $1$ over CM points of a fixed discriminant, twisted by a genus character) is well known, and follows from the fact that the classes of primitive binary quadratic forms of a fixed discriminant form a finite group, combined with classical ortogonality of characters. In our setting, classes of binary hermitian forms do not form a group. However, as explained in~\cite[Remark 5.9]{herreroimamogluvonpippichschwagenscheidt}, one can express~$\tr_{m,D}(1)$ as a multiple of an integral of a non-trivial adelic character on an adelic group, which implies the vanishing of~$\tr_{m,D}(1)$, again by orthogonality of characters. Note that in~\cite{herreroimamogluvonpippichschwagenscheidt} we worked with prime discriminants~$D$, but the properties used in the remark about the vanishing of~$\tr_{m,D}(1)$ extend easily to general fundamental discriminants. 
\end{remark}

This work is organized as follows. In Section~\ref{section niebur} we recall the Fourier expansion of the Niebur Poincar\'e series for $\PSL_2(\Z)$, which is given in terms of Bessel functions and Kloosterman sums. In Section~\ref{section exponential sums} we give an identity between these Kloosterman sums and certain exponential sums occuring in the Fourier expansion of the Niebur Poincar\'e series $F_{\nu}(P,s)$ on hyperbolic $3$-space. In Section~\ref{section proofs} we put these two results together to prove our main result Theorem~\ref{theorem trace niebur}. The proof of Theorem~\ref{theorem trace eisenstein} is presented in Section~\ref{section proofs 2}. Finally, in Section~\ref{sec: generating functions over nu} we recall the properties of the automorphic Green's function for~$\PSL_2(\mathcal{O}_D)$ and prove Theorem~\ref{thm generating function over nu} by using a result of Friedberg~\cite{friedberg} on raising operators for automorphic forms in~$\H^3$. 

%
%
%

\section{Niebur Poincar\'e series for $\PSL_2(\Z)$}\label{section niebur} We recall some known results about Niebur Poincar\'e series for $\PSL_2(\Z)$, following \cite{bruinieryang}. For a positive integer $n > 0$ it is defined by $(\tau = u+iv \in \H^2)$
\[
F_n(\tau,s) = \pi \sqrt{n}\sum_{M \in \Gamma'_\infty\backslash \Gamma'}\sqrt{v} I_{s-\frac{1}{2}}(2\pi n v)e(-nu)|_0 M,
\]
with the usual $I$-Bessel function, $\Gamma' = \PSL_2(\Z)$ and $\Gamma'_\infty = \{\pm\left(\begin{smallmatrix}1 & b \\0 & 1 \end{smallmatrix}\right): b \in \Z\}$. Niebur \cite{niebur} showed that $F_n(\tau,s)$ can be analytically continued to $s = 1$ via its Fourier expansion, which is given as follows.

\begin{proposition}\label{proposition fourier expansion niebur}The Fourier expansion of the Niebur Poincar\'e series $F_n(\tau,s)$ is given by
\begin{align*}
F_n(\tau,s) &= (2\mathcal{I}_s(2\pi n v)+\mathcal{K}_s(2\pi n v))e(-nu)\\
&\quad + c_n(0,s)v^{1-s} + \sum_{m \in \Z\setminus \{0\}}c_n(m,s)\mathcal{K}_s(2\pi m v)e(mu),
\end{align*}
with the Bessel functions $
\mathcal{I}_s(y) = \sqrt{\frac{\pi |y|}{2}}I_{s-1/2}(|y|)$
and
$
\mathcal{K}_s(y) = \sqrt{\frac{2|y|}{\pi}}K_{s-1/2}(|y|),
$
and coefficients
\[
c_n(m,s) = \begin{dcases}
2\pi \left| \frac{n}{m}\right|^{1/2}\sum_{c=1}^\infty H_c(m,n)I_{2s-1}\left(\frac{4\pi}{c}\sqrt{|mn|} \right) , & m > 0, \\
\frac{4\pi^{1+s}n^s}{(2s-1)\Gamma(s)}\sum_{c=1}^\infty c^{1-2s}H_c(n,0), & m = 0, \\
 -\delta_{-n,m}+2\pi \left| \frac{n}{m}\right|^{1/2}\sum_{c=1}^\infty H_c(m,n)J_{2s-1}\left(\frac{4\pi}{c}\sqrt{|mn|} \right), & m < 0,
\end{dcases}
\]
with the Kloosterman sum
\begin{align}\label{kloosterman sum}
H_c(m,n) = \frac{1}{c}\sum_{d(c)^*}e\left( \frac{nd-md'}{c} \right),
\end{align}
where the sum runs over the multiplicative group~$(\Z/c\Z)^\ast$ and~$d'$ denotes the inverse of~$d$ in that group.
\end{proposition}

For $m = 0$ we have the more explicit formula
\begin{equation}\label{eq constant term F_n}
    c_n(0,s) = \frac{4\pi}{(2s-1)}\frac{n^{1-s}\sigma_{2s-1}(n)}{\pi^{-s}\Gamma(s)\zeta(2s)},
\end{equation}
with the divisor sum $\sigma_s(n) = \sum_{d \mid n}d^{s}$, compare \cite[Proposition~2.2]{bruinieryang}. Moreover, we have the special values
\begin{align*}
\mathcal{I}_1(y) &= \sinh(|y|), \qquad \mathcal{K}_1(y) = e^{-|y|}, \qquad 2\mathcal{I}_1(y)+\mathcal{K}_1(y) = e^{|y|}.
\end{align*}
Plugging in $s = 1$, we find that
\[
F_n(\tau,1) = j_n(\tau) +24 \sigma_1(n),
\]
where $j_n(\tau)$ is the unique weakly holomorphic modular form of weight $0$ for $\PSL_2(\Z)$ with Fourier expansion of the form $j_n(\tau) = q^{-n}+O(q)$. In particular, we have
\begin{align}
\begin{split}\label{eq cn and jn}
c_n(m,1) &= \begin{cases}
24\sigma_1(n), & m = 0, \\
c_{n}(m), &m > 0, \\
0, & m < 0,
\end{cases}
\end{split} 
\end{align}
where $c_{n}(m)$ denote the coefficients of $j_n$.

\begin{remark}
    Niebur also defined~$F_n(\tau,s)$ for negative index $n$. These functions are related by the equality $F_{-n}(\tau,s)=F_n(-\overline{\tau},s)$. In particular, plugging $s=1$ we get 
    $$F_n(\tau,1)=j_{|n|}(-\overline{\tau})+24\sigma_1(|n|) \text{ when $n<0$},$$
    which is an anti-holomorphic modular function for $\PSL_2(\Z)$.
\end{remark}

\section{Exponential sums}\label{section exponential sums} Let $D < 0$ be a negative fundamental discriminant. For a natural number $c$ we let $D_c \mid D$ be the fundamental discriminant dividing $D$ which has the same prime divisors as $\text{g.c.d.}(c,D)$, that is, $D_c$ is the product of the prime discriminants $p^*$ for $p \mid \text{g.c.d.}(c,D)$. Note that $D/D_c$ is also a fundamental discriminant. Following \cite{bruinieryang}, for $\nu \in \mathfrak{d}_D^{-1}$ we consider the finite exponential sum
	\begin{align}\label{eq exponential sum}
	\widetilde{G}_c(|D|m,\nu) = \left( \frac{D/D_c}{c}\right)\sum_{\substack{b \in \mathcal{O}_D/c\mathcal{O}_D \\ |b|^2 \equiv -|D|m (c)}}\left(\frac{D_c}{\frac{|D|m+|b|^2}{c}}\right)e\left(\tr(\nu b)/c \right).
	\end{align}
	We have the following relation with the Kloosterman sum $H_c(m,n)$ given in \eqref{kloosterman sum}.

\begin{lemma}\label{lemma exponential sums}
	For $c \in \N$, $m \in \Z$ and $\nu \in \mathfrak{d}_D^{-1}$ we have
	\[
	\frac{1}{c}\widetilde{G}_c(|D|m,\nu) = \sum_{\substack{d \mid \nu \\ d \mid c}}\left( \frac{D}{d}\right)H_{c/d}(m,|D| |\nu|^2/d^2).
	\]
\end{lemma}

\begin{proof}
	This can be proved similarly as the Proposition in \cite[Section 4]{zagierdoinaganuma}. We also refer to \cite[Lemma~4.3]{bruinieryang} for the proof in the case that $D > 1$ is a positive odd prime discriminant, in which case one replaces $|b|^2$ with the norm $N(b)=bb'$ in $\Q(\sqrt{D})$. We leave the details to the reader. 
\end{proof}



\section{Twisted traces of Niebur Poincar\'e series - Proof of Theorem~\ref{theorem trace niebur}}\label{section proofs}

Theorem~\ref{theorem trace niebur} follows from the following evaluation of the twisted traces of the Niebur Poincar\'e series $F_{\nu}(P,s)$ on $\H^3$, by plugging in $s = 1$ and using \eqref{eq cn and jn}.

\begin{proposition}\label{proposition traces Niebur with s}
For $\nu \in \mathfrak{d}_D^{-1}$ with $\nu \neq 0$ we have
\[
\tr_{m,D}(F_{\nu}(\, \cdot \, , s)) = m\sum_{d \mid \nu}\left( \frac{D}{d}\right)d \, c_{|D||\nu|^2/d^2}(m,\tfrac{s+1}{2}), 
\]
where $c_n(m,s)$ are the coefficients of the Niebur Poincar\'e series $F_n(\tau,s)$ on $\PSL_2(\Z)$ from Proposition~\ref{proposition fourier expansion niebur}.
\end{proposition}

\begin{proof}
	We first write
	\begin{align*}
	\tr_{m,D}(F_{\nu}(\, \cdot \, , s)) &= 2\pi |\nu|\sum_{X \in \Gamma \backslash L_{|D|m}^+}\frac{\chi_D(X)}{|\Gamma_X|}\sum_{\gamma \in \Gamma_\infty \backslash \Gamma}r(\gamma P_X)I_s(4\pi |\nu|r(\gamma P_X))e(\tr(\nu z(\gamma P_X))) \\
	&= 2\pi |\nu|\sum_{X \in \Gamma_\infty \backslash L_{|D|m}^+}\chi_D(X)r(P_X)I_s(4\pi |\nu|r(P_X))e(\tr(\nu z(P_X))).
	\end{align*}
	For~$X=\left(\begin{smallmatrix}a & b \\ \overline{b} & c \end{smallmatrix}\right)\in L^+_{m|D|}$ we have $r(P_X) = \frac{\sqrt{|D|m}}{c}$ and $z(P_X) = \frac{b}{c}$. Moreover, note that $\left(\begin{smallmatrix}1 & \beta \\ 0 & 1 \end{smallmatrix} \right) \in \Gamma_\infty$ changes $b$ to $b+a\beta$, but does not change $c$. Hence, a system of representatives for $\Gamma_\infty \backslash L_{|D|m}^+$ is given by the matrices $\left(\begin{smallmatrix}a & b \\ \overline{b} & c \end{smallmatrix}\right)$ where $c$ runs through all positive integers, $b$ runs through $\mathcal{O}_D/c\mathcal{O}_D$ with $|b|^2 \equiv - |D|m \pmod{c}$, and $a$ is determined by $a = \frac{|D|m + |b|^2}{c}$. Then we have
	\[
	\chi_D\left(\begin{pmatrix}a & b \\ \overline{b} & c \end{pmatrix} \right) = \left(\frac{D/D_c}{c} \right)\left( \frac{D_c}{a}\right) = \left(\frac{D/D_c}{c} \right)\left( \frac{D_c}{\frac{|D|m+|b|^2}{c}}\right).
	\]
	Hence we get 
	\begin{align*}
	& \tr_{m,D}(F_{\nu}(\, \cdot \, , s)) \\
	&\quad= 2\pi |\nu|\sum_{c > 0}\sum_{\substack{b \in \mathcal{O}_D / c\mathcal{O}_D \\ |b|^2 \equiv - |D|m (c)}}\left( \frac{D/D_c}{c}\right)\left(\frac{D_c}{\frac{|D|m+|b|^2}{c}} \right)\frac{\sqrt{|D|m}}{c}I_s\left(4\pi |\nu|\frac{\sqrt{|D|m}}{c}\right)e(\tr(\nu b)/c) \\
	&\quad = 2\pi |\nu| \sqrt{|D|m}\sum_{c > 0}\frac{1}{c}\widetilde{G}_c(|D|m,\nu)I_s\left(4\pi |\nu|\frac{\sqrt{|D|m}}{c}\right),
	\end{align*}
	with the exponential sum $\widetilde{G}_c(|D|m,\nu)$ defined in \eqref{eq exponential sum}. Using Lemma~\ref{lemma exponential sums} we can further compute
	\begin{align*}
	\tr_{m,D}(F_{\nu}(\, \cdot \, , s))&= 2\pi |\nu|\sqrt{|D|m}\sum_{c > 0}\sum_{\substack{d \mid \nu \\ d \mid c}}\left( \frac{D}{d}\right)H_{c/d}(m,|D||\nu|^2/d^2)I_s\left(4\pi |\nu|\frac{\sqrt{|D|m}}{c}\right) \\
	& = 2\pi |\nu|\sqrt{|D|m}\sum_{\substack{d \mid \nu}}\left( \frac{D}{d}\right)\sum_{c > 0}H_{c}(m,|D||\nu|^2/d^2)I_s\left(4\pi |\nu|\frac{\sqrt{|D|m}}{cd}\right).
	\end{align*}
	Comparing the series over $c > 0$ to the coefficients of the Niebur Poincar\'e series from Proposition~\ref{proposition fourier expansion niebur}, we obtain
	\begin{align*}
	\tr_{m,D}(F_{\nu}(\, \cdot \, , s)) &= m\sum_{\substack{d \mid \nu}}\left( \frac{D}{d}\right) d \, c_{|D||\nu|^2/d^2}(m,\tfrac{s+1}{2}).
	\end{align*}
	This gives the stated formula.
\end{proof}


\section{Twisted traces of Eisenstein series - Proof of Theorem~\ref{theorem trace eisenstein}}\label{section proofs 2}

Replicating the first steps from the proof of Proposition~\ref{proposition traces Niebur with s} in Section~\ref{section proofs} we arrive at

\begin{align*}
	\tr_{m,D}(E_{0}(\, \cdot \,, s)) &= \sum_{c > 0}\sum_{\substack{b \in \mathcal{O}_F/c\mathcal{O}_F \\ |b|^2 \equiv -|D|m (c)}}\left(\frac{D/D_c}{c} \right)\left( \frac{D_c}{\frac{|D|m+|b|^2}{c}}\right)\left(\frac{\sqrt{|D|m}}{c} \right)^{s+1}\\
	&= \left(\sqrt{|D|m}\right)^{s+1}\sum_{c > 0}c^{-s-1}\widetilde{G}_c(|D|m,0).
	\end{align*}
	Using Lemma~\ref{lemma exponential sums} we get
	\begin{align*}
	\tr_{m,D}(E_{0}(\, \cdot \,, s)) &= \left(\sqrt{|D|m}\right)^{s+1}\sum_{c > 0}c^{-s}\sum_{r \mid c}\left( \frac{D}{r}\right)H_{c/r}(m,0) \\
	&= \left(\sqrt{|D|m}\right)^{s+1}\sum_{t > 0}\sum_{r > 0}(tr)^{-s}\left( \frac{D}{r}\right)H_{t}(m,0) \\
	&= \left(\sqrt{|D|m}\right)^{s+1}L_D(s)\sum_{t > 0}t^{-s}H_{t}(m,0).
\end{align*}
	Comparing the series on the right-hand side to the coefficient $c_m(0,s)$ of the Niebur Poincar\'e series $F_m(\tau,s)$ from Proposition~\ref{proposition fourier expansion niebur}, and using \eqref{eq constant term F_n}, we get
	\begin{align*}
	\tr_{m,D}(E_{0}(\, \cdot \,, s)) &=\left(\sqrt{|D|m}\right)^{s+1}L_D(s)\frac{s\Gamma(\tfrac{s+1}{2})}{4\pi^{\frac{s+3}{2}}m^{\frac{s+1}{2}}}c_{m}(0,\tfrac{s+1}{2}) \\
	&= \left(\sqrt{|D|m}\right)^{s+1}L_D(s)\frac{s\Gamma(\tfrac{s+1}{2})}{4\pi^{\frac{s+3}{2}}m^{\frac{s+1}{2}}}\frac{4\pi}{s}\frac{m^{\frac{1-s}{2}}\sigma_{s}(m)}{\pi^{-\frac{s+1}{2}}\Gamma(\tfrac{s+1}{2})\zeta(s+1)} \\
	&= \frac{|D|^{\frac{s+1}{2}}L_D(s)}{\zeta(s+1)}m^{\frac{1-s}{2}}\sigma_s(m).
	\end{align*}
	This proves  the stated formula. At $s = 1$  using Dirichlet's class number formula, $L_D(1) = \frac{2\pi h(D)}{\sqrt{|D|}w(D)}$ and $\zeta(2) = \frac{\pi^2}{6}$ we complete the proof of Theorem~\ref{theorem trace eisenstein}.

\section{Generating function over~$\nu$ - Proof of Theorem~\ref{thm generating function over nu}}\label{sec: generating functions over nu}

Corollary~\ref{cor: generating series over m} gives the modularity of the generating function for the twisted traces~$\tr_{m,D}(J_\nu)$ when summed over~$m$, for fixed~$\nu\in \mathfrak{d}_D^{-1}$ with~$\nu\neq 0$. Instead in this section we fix~$m$, and summing over $\nu$ we prove Theorem~\ref{thm generating function over nu}, which gives a weight 2 automorphic form~$\mathcal{F}_{m,D}=(\mathcal{F}_{m,D}^{(2)},\mathcal{F}_{m,D}^{(1)},\mathcal{F}_{m,D}^{(0)})^t$ on~$\H^3\setminus T_{m,D}$, whose Fourier coefficients are given in terms of ~$\tr_{m,D}(J_\nu)$. 

We start by recalling the construction of the automorphic Green's function for $\Gamma$.  It is defined by\footnote{Here we follow~\cite{herreroimamogluvonpippichschwagenscheidt} where we used a different normalization than \cite{elstrodt} and \cite{herreroimamogluvonpippichtoth}.}
\begin{equation*}\label{eq Green's function}
G_s(P_1,P_2) = \pi\sum_{\gamma \in \Gamma}\varphi_s\big(\cosh(d(P_1,\gamma P_2))\big),
\end{equation*}
for~$P_1,P_2\in \H^3, P_1\not \equiv P_2$ (mod~$\Gamma$) and~$s\in \C$ with $\Re(s)>1$. Here
\begin{equation*}\label{eq:varphi_s}
\varphi_s(t) = \left(t+\sqrt{t^2-1}\right)^{-s} (t^2-1)^{-1/2} \qquad \text{for }t>1.
\end{equation*}
The Green's function is~$\Gamma$-invariant in each variable and symmetric in~$P_1,P_2$. It defines a smooth function on~$(\Gamma \backslash \H^3)\times (\Gamma \backslash \H^3)$ away from the diagonal, with a singularity of the form
$$G_s(P,Q)=\frac{\pi |\Gamma_Q|}{d(P,Q)}+O_Q(1) \text{ as }P\to Q,$$
where~$d(\cdot,\cdot)$ denotes the hyperbolic distance. Moreover, it satisfies
\[
(\Delta_{P_1} - (1-s^2)) G_s(P_1,P_2) = 0.
\]
As a function of~$s$ it has  meromorphic continuation to~$\C$ with~$s=1$ a simple pole with constant residue (independent of~$P_1$ and~$P_2$).

The Fourier expansion of the Green's function in the variable~$P_2=z_2+r_2j$, for~$r_2>r(\gamma P_1)$ for all~$\gamma \in \Gamma$, is given by
$$G_s(P_1,P_2)=\frac{4\pi^2}{\sqrt{|D|}}\Bigg(\frac{r_2^{1-s}}{s}E_0(P_1,s)+\frac{1}{\pi}\sum_{\substack{\nu \in \mathfrak{d}^{-1}_D \\ \nu\neq 0}}F_{-\nu}(P_1,s)|\nu|^{-1}r_2K_s(4\pi |\nu|r_2)e(\tr(\nu z_2))\Bigg).$$
This Fourier expansion has analytic continuation to~$\mathrm{Re}(s)>\frac{1}{2},s\neq 1$. See~\cite{elstrodt} and ~\cite{herreroimamogluvonpippichtoth} for proofs.

It follows from the properties of the Green's function described above that the function
$$\mathcal{L}_{m,D}(P)=\lim_{s\to 1}\tr_{m,D}\left(\frac{\sqrt{|D|}}{4\pi}G_s(\cdot,P)-\frac{\pi r(P)^{1-s}}{s} E_0(\cdot,s)\right),$$
for~$P\in \H^3\setminus T_{m,D}$, is~$\Gamma$-invariant, harmonic, and has Fourier expansion
\begin{equation}\label{eq L_{m,D}}
    \mathcal{L}_{m,D}(z+rj)=
\sum_{\substack{\nu \in \mathfrak{d}^{-1}_D \\ \nu\neq 0}}\tr_{m,D}(J_{\nu})|\nu|^{-1}rK_1(4\pi |\nu|r)e(\tr(\nu z)),
\end{equation}
valid when~$r>r(Q)$ for all~$Q\in T_{m,D}$ (in particular, when~$r>\sqrt{m|D|}$). Finally, by~\cite[Proposition~1.2]{friedberg}, the function~$\mathfrak{D}_0\mathcal{L}_{m,D}$, where~$\mathfrak{D}_0$ denotes the raising operator
$$\mathfrak{D}_0=\frac{1}{2\pi i}(-\partial_z,\partial_r,\partial_{\overline{z}})^t,$$
is a smooth automorphic form of weight 2 for~$\Gamma$ on~$\H^3\setminus T_{m,D}$. Finally, applying~$\mathfrak{D}_0$ to~\eqref{eq L_{m,D}}, using 
the identity
$$\frac{\partial}{\partial y} K_1(y)=-K_0-\frac{1}{y}K_1(y)$$
(see, e.g., \cite[Formula~8.472(1)]{GRtableofintegrals}) and comparing the result with~\eqref{eq F_i}, 
we conclude that~$\mathcal{F}_{m,D}=\mathfrak{D}_0\mathcal{L}_{m,D}$. This proves Theorem~\ref{thm generating function over nu}.

\end{document}